\documentclass{article}
\usepackage{amsmath,amsthm,amsopn,amstext,amscd,amsfonts,amssymb,verbatim}
\usepackage{pdflscape}
\usepackage{natbib}

\def \u {\mathop{\rm \mathcal{U}}\nolimits}
\def \tr {\mathop{\rm tr}\nolimits}
\def \re {\mathop{\rm Re}\nolimits}

\def \eig {\mathop{\rm eig}\nolimits}

\def \Vol {\mathop{\rm Vol}\nolimits}

\def \etr {\mathop{\rm etr}\nolimits}
\def \diag {\mathop{\rm diag}\nolimits}

\renewenvironment{abstract}
                 {\vspace{6pt}
                  \begin{center}
                  \begin{minipage}{5in}
                  \centerline{\textbf{Abstract}}
                  \noindent\ignorespaces
                 }
                 {\end{minipage}\end{center}}
\newtheorem{theorem}{\textbf{Theorem}}[section]

\newtheorem{proposition}{\textbf{Proposition}}[section]
\theoremstyle{definition}
\newtheorem{definition}{\textbf{Definition}}[section]

\newtheorem{remark}{\textbf{Remark}}[section]

\title{\Large \textbf{Generalised matrix multivariate $T$-distribution}}
\author{
  \textbf{Jos\'e A. D\'{\i}az-Garc\'{\i}a} \thanks{Corresponding author\newline
    }\\
  {\normalsize Department of Statistics and Computation} \\
  {\normalsize 25350 Buenavista, Saltillo, Coahuila, Mexico} \\
  {\normalsize E-mail: jadiaz@uaaan.mx} \\[2ex]
  \textbf{Ram\'on Guti\'errez-S\'anchez} \\
  {\normalsize Department of Statistics and O.R} \\
  {\normalsize University of Granada} \\
  {\normalsize Granada 18071, Spain}\\
  {\normalsize E-mail: ramongs@ugr.es}\\
}
\date{}
\begin{document}
\maketitle

\begin{abstract}
Supposing Kotz-Riesz type I and II distributions and their corresponding independent univariate
Riesz distributions the associated generalised matrix multivariate $T$ distributions, termed
matrix multivariate $T$-Riesz distributions are obtained. In addition, its various properties
are studied. All these results are obtained for real normed division algebras.
\end{abstract}

\section{Introduction}\label{sec1}

In many statistical models, as an alternative to the use of matrix multivariate normal distribution from
the 80's  it has been assumed a matrix multivariate elliptical distribution. Actually, the matrix
multivariate elliptical distribution is a family of distributions that includes the matrix multivariate
normal, contaminated normal, Pearson type II and VII, Kotz, Jensen-Logistic, power exponential and Bessel
distributions, among others. These distributions have tails that are more or less weighted, and/or
display a greater or smaller degree of kurtosis than the normal distribution, refer to \citet{fz:90} and
\citet{gv:93}.

In addition, matrix multivariate elliptical distributions are of great interest due to the next
invariance property: Assume that $\mathbf{X}$ is distributed according to a matrix multivariate
distribution, then the distributions of certain type of matrix transformations of the random matrix, say
$\mathbf{Y} = f(\mathbf{X})$, are invariant under all class of matrix multivariate elliptical
distribution, furthermore, such distributions coincide when $\mathbf{X}$ is normally assumed, see
\citet{fz:90} and \citet{gv:93}.

However, this invariance property is present when certain statistical (probabilistic) dependence is
assumed. For example, if $\mathbf{X} = \left[
\begin{array}{c}
  \mathbf{X}_{1} \\
  \mathbf{X}_{2}
\end{array}
\right]$ has a matrix multivariate elliptical distribution, then $\mathbf{X}_{1}$ and $\mathbf{X}_{2}$
are statistically dependent, observing that $\mathbf{X}_{1}$ and $\mathbf{X}_{2}$ are probabilistically
independent if $\mathbf{X}$ has a matrix multivariate normal distribution, \citet{gv:93}. Then, if is
defined $\mathbf{T} = \mathbf{X}_{1}(\mathbf{X}'_{2}\mathbf{X}_{2})^{-1/2}$, where $\mathbf{X}'$ denotes
the transpose of $\mathbf{X}$, it is said that $\mathbf{T}$ has a matrix multivariate $T$-distribution,
and its distribution is the same under all matrix multivariate elliptical distribution and this coincides
with the distribution obtained when $\mathbf{X}$ follow a matrix multivariate normal distribution.

The independent case cited above can be found in the Bayesian inference, see \citet{p:82}. In particular,
assume that certain distribution is function of two matrix parameters, say $\boldsymbol{\delta}_{1}$,
$\boldsymbol{\delta}_{2}$ for which, it is suppose that their prior distributions  belong to the class of
matrix variate elliptical distribution and are independent. Then, is of interest find the prior
distribution of a parameter type $\mathbf{T}$ defined as
$\boldsymbol{\delta}_{1}(\boldsymbol{\delta}'_{2}\boldsymbol{\delta}_{2})^{-1/2}$. In this case the
distribution of $\mathbf{T}$ is different for each particular elliptical distribution.

A distribution of particular interest is the matrix multivariate elliptical distribution termed
Kotz-Riesz distribution. This interest is based in the relation with the Riesz distribution,
\citet{dg:15c}. If $\mathbf{X}$ is distributed according a matrix multivariate Kotz-Riesz, then the
matrix $\mathbf{V} = \mathbf{X}'\mathbf{X}$ has a Riesz distribution. The Riesz distributions, was first
introduced by \citet{hl:01} under the name of Riesz natural exponential family (Riesz NEF); it was based
on a special case of the so-called Riesz measure from \citet[p.137]{fk:94}. This Riesz distribution
generalises the matrix multivariate gamma and Wishart distributions, containing them as particular cases.

In analogy with the case of $T$-distribution under normality , exist two possible generalisations of it
when a Kotz-Riesz distribution is assumed, see \citet{dggj:12}. In this paper is addressed the case of
the distribution termed matrix multivariate $T$-Riesz distribution.

This present article is organised as follow; some basic concepts and the notation of abstract algebra and
Jacobians are summarised in Section \ref{sec2}. The nonsingular central matrix multivariate $T$-Riesz
type I and II distributions and the corresponding generalised matrix multivariate beta type II
distributions  are studied in Section \ref{sec3}. Finally, the joint densities of the singular values are
derived in Section \ref{sec4}. All these results are derived for real normed division algebras.

\section{Preliminary results}\label{sec2}

A detailed discussion of real normed division algebras can be found in \citet{b:02} and \citet{E:90}. For
your convenience, we shall introduce some notation, although in general, we adhere to standard notation
forms.

For our purposes: Let $\mathbb{F}$ be a field. An \emph{algebra} $\mathfrak{A}$ over $\mathbb{F}$ is a
pair $(\mathfrak{A};m)$, where $\mathfrak{A}$ is a \emph{finite-dimensional vector space} over
$\mathbb{F}$ and \emph{multiplication} $m : \mathfrak{A} \times \mathfrak{A} \rightarrow A$ is an
$\mathbb{F}$-bilinear map; that is, for all $\lambda \in \mathbb{F},$ $x, y, z \in \mathfrak{A}$,
\begin{eqnarray*}
  m(x, \lambda y + z) &=& \lambda m(x; y) + m(x; z) \\
  m(\lambda x + y; z) &=& \lambda m(x; z) + m(y; z).
\end{eqnarray*}
Two algebras $(\mathfrak{A};m)$ and $(\mathfrak{E}; n)$ over $\mathbb{F}$ are said to be
\emph{isomorphic} if there is an invertible map $\phi: \mathfrak{A} \rightarrow \mathfrak{E}$ such that
for all $x, y \in \mathfrak{A}$,
$$
  \phi(m(x, y)) = n(\phi(x), \phi(y)).
$$
By simplicity, we write $m(x; y) = xy$ for all $x, y \in \mathfrak{A}$.

Let $\mathfrak{A}$ be an algebra over $\mathbb{F}$. Then $\mathfrak{A}$ is said to be
\begin{enumerate}
  \item \emph{alternative} if $x(xy) = (xx)y$ and $x(yy) = (xy)y$ for all $x, y \in \mathfrak{A}$,
  \item \emph{associative} if $x(yz) = (xy)z$ for all $x, y, z \in \mathfrak{A}$,
  \item \emph{commutative} if $xy = yx$ for all $x, y \in \mathfrak{A}$, and
  \item \emph{unital} if there is a $1 \in \mathfrak{A}$ such that $x1 = x = 1x$ for all $x \in \mathfrak{A}$.
\end{enumerate}
If $\mathfrak{A}$ is unital, then the identity 1 is uniquely determined.

An algebra $\mathfrak{A}$ over $\mathbb{F}$ is said to be a \emph{division algebra} if $\mathfrak{A}$ is
nonzero and $xy = 0_{\mathfrak{A}} \Rightarrow x = 0_{\mathfrak{A}}$ or $y = 0_{\mathfrak{A}}$ for all
$x, y \in \mathfrak{A}$.

The term ``division algebra", comes from the following proposition, which shows that, in such an algebra,
left and right division can be unambiguously performed.

Let $\mathfrak{A}$ be an algebra over $\mathbb{F}$. Then $\mathfrak{A}$ is a division algebra if, and
only if, $\mathfrak{A}$ is nonzero and for all $a, b \in \mathfrak{A}$, with $b \neq 0_{\mathfrak{A}}$,
the equations $bx = a$ and $yb = a$ have unique solutions $x, y \in \mathfrak{A}$.

In the sequel we assume $\mathbb{F} = \Re$ and consider classes of division algebras over $\Re$ or
``\emph{real division algebras}" for short.

We introduce the algebras of \emph{real numbers} $\Re$, \emph{complex numbers} $\mathfrak{C}$,
\emph{quaternions} $\mathfrak{H}$ and \emph{octonions} $\mathfrak{O}$. Then, if $\mathfrak{A}$ is an
alternative real division algebra, then $\mathfrak{A}$ is isomorphic to $\Re$, $\mathfrak{C}$,
$\mathfrak{H}$ or $\mathfrak{O}$.

Let $\mathfrak{A}$ be a real division algebra with identity $1$. Then $\mathfrak{A}$ is said to be
\emph{normed} if there is an inner product $(\cdot, \cdot)$ on $\mathfrak{A}$ such that
$$
  (xy, xy) = (x, x)(y, y) \qquad \mbox{for all } x, y \in \mathfrak{A}.
$$
If $\mathfrak{A}$ is a \emph{real normed division algebra}, then $\mathfrak{A}$ is isomorphic $\Re$,
$\mathfrak{C}$, $\mathfrak{H}$ or $\mathfrak{O}$.

There are exactly four normed division algebras: real numbers ($\Re$), complex numbers ($\mathfrak{C}$),
quaternions ($\mathfrak{H}$) and octonions ($\mathfrak{O}$), see \citet{b:02}. We take into account that
should be taken into account, $\Re$, $\mathfrak{C}$, $\mathfrak{H}$ and $\mathfrak{O}$ are the only
normed division algebras; furthermore, they are the only alternative division algebras.

Let $\mathfrak{A}$ be a division algebra over the real numbers. Then $\mathfrak{A}$ has dimension either
1, 2, 4 or 8. In other branches of mathematics, the parameters $\alpha = 2/\beta$ and $t = \beta/4$ are
used, see \citet{er:05} and \citet{k:84}, respectively.

Finally, observe that

\begin{tabular}{c}
  $\Re$ is a real commutative associative normed division algebras, \\
  $\mathfrak{C}$ is a commutative associative normed division algebras,\\
  $\mathfrak{H}$ is an associative normed division algebras, \\
  $\mathfrak{O}$ is an alternative normed division algebras. \\
\end{tabular}

Let $\mathfrak{L}^{\beta}_{n,m}$ be the set of all $n \times m$ matrices of rank $m \leq n$ over
$\mathfrak{A}$ with $m$ distinct positive singular values, where $\mathfrak{A}$ denotes a \emph{real
finite-dimensional normed division algebra}. Let $\mathfrak{A}^{n \times m}$ be the set of all $n \times
m$ matrices over $\mathfrak{A}$. The dimension of $\mathfrak{A}^{n \times m}$ over $\Re$ is $\beta mn$.
Let $\mathbf{A} \in \mathfrak{A}^{n \times m}$, then $\mathbf{A}^{*} = \bar{\mathbf{A}}^{T}$ denotes the
usual conjugate transpose.

Table \ref{table1} sets out the equivalence between the same concepts in the four normed division
algebras.

\begin{table}[th]
  \centering
  \caption{\scriptsize Notation}\label{table1}
  \begin{scriptsize}
  \begin{tabular}{cccc|c}
    \hline
    Real & Complex & Quaternion & Octonion & \begin{tabular}{c}
                                               Generic \\
                                               notation \\
                                             \end{tabular}\\
    \hline
    Semi-orthogonal & Semi-unitary & Semi-symplectic & \begin{tabular}{c}
                                                         Semi-exceptional \\
                                                         type \\
                                                       \end{tabular}
      & $\mathcal{V}_{m,n}^{\beta}$ \\
    Orthogonal & Unitary & Symplectic & \begin{tabular}{c}
                                                         Exceptional \\
                                                         type \\
                                                       \end{tabular} & $\mathfrak{U}^{\beta}(m)$ \\
    Symmetric & Hermitian & \begin{tabular}{c}
                              Quaternion \\
                              hermitian \\
                            \end{tabular}
     & \begin{tabular}{c}
                              Octonion \\
                              hermitian \\
                            \end{tabular} & $\mathfrak{S}_{m}^{\beta}$ \\
    \hline
  \end{tabular}
  \end{scriptsize}
\end{table}

We denote by ${\mathfrak S}_{m}^{\beta}$ the real vector space of all $\mathbf{S} \in \mathfrak{A}^{m
\times m}$ such that $\mathbf{S} = \mathbf{S}^{*}$. In addition, let $\mathfrak{P}_{m}^{\beta}$ be the
\emph{cone of positive definite matrices} $\mathbf{S} \in \mathfrak{A}^{m \times m}$. Thus,
$\mathfrak{P}_{m}^{\beta}$ consist of all matrices $\mathbf{S} = \mathbf{X}^{*}\mathbf{X}$, with
$\mathbf{X} \in \mathfrak{L}^{\beta}_{n,m}$; then $\mathfrak{P}_{m}^{\beta}$ is an open subset of
${\mathfrak S}_{m}^{\beta}$.

Let $\mathfrak{D}_{m}^{\beta}$ consisting of all $\mathbf{D} \in \mathfrak{A}^{m \times m}$, $\mathbf{D}
= \diag(d_{1}, \dots,d_{m})$. Let $\mathfrak{T}_{U}^{\beta}(m)$ be the subgroup of all \emph{upper
triangular} matrices $\mathbf{T} \in \mathfrak{A}^{m \times m}$ such that $t_{ij} = 0$ for $1 < i < j
\leq m$.

For any matrix $\mathbf{X} \in \mathfrak{A}^{n \times m}$, $d\mathbf{X}$ denotes the\emph{ matrix of
differentials} $(dx_{ij})$. Finally, we define the \emph{measure} or volume element $(d\mathbf{X})$ when
$\mathbf{X} \in \mathfrak{A}^{n \times m}, \mathfrak{S}_{m}^{\beta}$, $\mathfrak{D}_{m}^{\beta}$ or
$\mathcal{V}_{m,n}^{\beta}$, see \citet{dggj:11} and \citet{dggj:13}.

If $\mathbf{X} \in \mathfrak{A}^{n \times m}$ then $(d\mathbf{X})$ (the Lebesgue measure in
$\mathfrak{A}^{n \times m}$) denotes the exterior product of the $\beta mn$ functionally independent
variables
$$
  (d\mathbf{X}) = \bigwedge_{i = 1}^{n}\bigwedge_{j = 1}^{m}dx_{ij} \quad \mbox{ where }
    \quad dx_{ij} = \bigwedge_{k = 1}^{\beta}dx_{ij}^{(k)}.
$$

If $\mathbf{S} \in \mathfrak{S}_{m}^{\beta}$ (or $\mathbf{S} \in \mathfrak{T}_{U}^{\beta}(m)$ with
$t_{ii} >0$, $i = 1, \dots,m$) then $(d\mathbf{S})$ (the Lebesgue measure in $\mathfrak{S}_{m}^{\beta}$
or in $\mathfrak{T}_{U}^{\beta}(m)$) denotes the exterior product of the exterior product of the
$m(m-1)\beta/2 + m$ functionally independent variables,
$$
  (d\mathbf{S}) = \bigwedge_{i=1}^{m} ds_{ii}\bigwedge_{i > j}^{m}\bigwedge_{k = 1}^{\beta}
                      ds_{ij}^{(k)}.
$$
Observe, that for the Lebesgue measure $(d\mathbf{S})$ defined thus, it is required that $\mathbf{S} \in
\mathfrak{P}_{m}^{\beta}$, that is, $\mathbf{S}$ must be a non singular Hermitian matrix (Hermitian
definite positive matrix).

If $\mathbf{\Lambda} \in \mathfrak{D}_{m}^{\beta}$ then $(d\mathbf{\Lambda})$ (the Legesgue measure in
$\mathfrak{D}_{m}^{\beta}$) denotes the exterior product of the $\beta m$ functionally independent
variables
$$
  (d\mathbf{\Lambda}) = \bigwedge_{i = 1}^{n}\bigwedge_{k = 1}^{\beta}d\lambda_{i}^{(k)}.
$$
If $\mathbf{H}_{1} \in \mathcal{V}_{m,n}^{\beta}$ then
$$
  (\mathbf{H}^{*}_{1}d\mathbf{H}_{1}) = \bigwedge_{i=1}^{m} \bigwedge_{j =i+1}^{n}
  \mathbf{h}_{j}^{*}d\mathbf{h}_{i}.
$$
where $\mathbf{H} = (\mathbf{H}^{*}_{1}|\mathbf{H}^{*}_{2})^{*} = (\mathbf{h}_{1}, \dots,
\mathbf{h}_{m}|\mathbf{h}_{m+1}, \dots, \mathbf{h}_{n})^{*} \in \mathfrak{U}^{\beta}(n)$. It can be
proved that this differential form does not depend on the choice of the $\mathbf{H}_{2}$ matrix. When $n
= 1$; $\mathcal{V}^{\beta}_{m,1}$ defines the unit sphere in $\mathfrak{A}^{m}$. This is, of course, an
$(m-1)\beta$- dimensional surface in $\mathfrak{A}^{m}$. When $n = m$ and denoting $\mathbf{H}_{1}$ by
$\mathbf{H}$, $(\mathbf{H}d\mathbf{H}^{*})$ is termed the \emph{Haar measure} on
$\mathfrak{U}^{\beta}(m)$.

The surface area or volume of the Stiefel manifold $\mathcal{V}^{\beta}_{m,n}$ is
\begin{equation}\label{vol}
    \Vol(\mathcal{V}^{\beta}_{m,n}) = \int_{\mathbf{H}_{1} \in
  \mathcal{V}^{\beta}_{m,n}} (\mathbf{H}_{1}d\mathbf{H}^{*}_{1}) =
  \frac{2^{m}\pi^{mn\beta/2}}{\Gamma^{\beta}_{m}[n\beta/2]},
\end{equation}
where $\Gamma^{\beta}_{m}[a]$ denotes the multivariate \emph{Gamma function} for the space
$\mathfrak{S}_{m}^{\beta}$. This can be obtained as a particular case of the \emph{generalised gamma
function of weight $\kappa$} for the space $\mathfrak{S}^{\beta}_{m}$ with $\kappa = (k_{1}, k_{2},
\dots, k_{m}) \in \Re^{m}$, taking $\kappa =(0,0,\dots,0) \in \Re^{m}$ and which for $\re(a) \geq
(m-1)\beta/2 - k_{m}$ is defined by, see \citet{gr:87} and \citet{fk:94},
\begin{eqnarray}\label{int1}
  \Gamma_{m}^{\beta}[a,\kappa] &=& \displaystyle\int_{\mathbf{A} \in \mathfrak{P}_{m}^{\beta}}
  \etr\{-\mathbf{A}\} |\mathbf{A}|^{a-(m-1)\beta/2 - 1} q_{\kappa}(\mathbf{A}) (d\mathbf{A}) \\
&=& \pi^{m(m-1)\beta/4}\displaystyle\prod_{i=1}^{m} \Gamma[a + k_{i}
    -(i-1)\beta/2]\nonumber\\ \label{gammagen1}
&=& [a]_{\kappa}^{\beta} \Gamma_{m}^{\beta}[a],
\end{eqnarray}
where $\etr(\cdot) = \exp(\tr(\cdot))$, $|\cdot|$ denotes the determinant, and for $\mathbf{A} \in
\mathfrak{S}_{m}^{\beta}$
\begin{equation}\label{hwv}
    q_{\kappa}(\mathbf{A}) = |\mathbf{A}_{m}|^{k_{m}}\prod_{i = 1}^{m-1}|\mathbf{A}_{i}|^{k_{i}-k_{i+1}}
\end{equation}
with $\mathbf{A}_{p} = (a_{rs})$, $r,s = 1, 2, \dots, p$, $p = 1,2, \dots, m$ is termed the \emph{highest
weight vector}, see \citet{gr:87}. Also,
\begin{eqnarray*}
  \Gamma_{m}^{\beta}[a] &=& \displaystyle\int_{\mathbf{A} \in \mathfrak{P}_{m}^{\beta}}
  \etr\{-\mathbf{A}\} |\mathbf{A}|^{a-(m-1)\beta/2 - 1}(d\mathbf{A}) \\ \label{cgamma}
    &=& \pi^{m(m-1)\beta/4}\displaystyle\prod_{i=1}^{m} \Gamma[a-(i-1)\beta/2],
\end{eqnarray*}
and $\re(a)> (m-1)\beta/2$.

In other branches of mathematics the \textit{highest weight vector} $q_{\kappa}(\mathbf{A})$ is also
termed the \emph{generalised power} of $\mathbf{A}$ and is denoted as $\Delta_{\kappa}(\mathbf{A})$, see
\citet{fk:94} and \citet{hl:01}.

Additional properties of $q_{\kappa}(\mathbf{A})$, which are immediate consequences of the definition of
$q_{\kappa}(\mathbf{A})$ are:
\begin{enumerate}
  \item Let $\mathbf{A} = \mathbf{L}^{*}\mathbf{DL}$ be the L'DL decomposition of $\mathbf{A} \in \mathfrak{P}_{m}^{\beta}$,
        where $\mathbf{L} \in \mathfrak{T}_{U}^{\beta}(m)$ with $l_{ii} = 1$, $i = 1, 2, \ldots ,m$ and
        $\mathbf{D} = \diag(\lambda_{1}, \dots, \lambda_{m})$, $\lambda_{i} \geq 0$, $i = 1, 2, \ldots
        ,m$. Then
        \begin{equation}\label{qk1}
          q_{\kappa}(\mathbf{A}) = \prod_{i=1}^{m} \lambda_{i}^{k_{i}}.
        \end{equation}
      \item
      \begin{equation}\label{qk2}
        q_{\kappa}(\mathbf{A}^{-1}) =  q_{-\kappa^{*}}^{*}(\mathbf{A}),
      \end{equation}
      where $\kappa^{*}=(k_{m}, k_{m-1}, \dots,k_{1})$, $-\kappa^{*}=(-k_{m}, -k_{m-1},
      \dots,-k_{1})$,
      \begin{equation}\label{hhwv}
         q_{\kappa}^{*}(\mathbf{A}) = |\mathbf{A}_{m}|^{k_{m}}\prod_{i = 1}^{m-1}|\mathbf{A}_{i}|^{k_{i}-k_{i+1}}
      \end{equation}
      and
      \begin{equation}\label{qqk1}
        q_{\kappa}^{*}(\mathbf{A}) = \prod_{i=1}^{m} \lambda_{i}^{k_{m-i+1}},
      \end{equation}
      see \citet[pp. 126-127 and Proposition VII.1.5]{fk:94}.

  Alternatively, let $\mathbf{A} = \mathbf{T}^{*}\mathbf{T}$ the Cholesky's decomposition of
  matrix $\mathbf{A} \in \mathfrak{P}_{m}^{\beta}$, with $\mathbf{T}=(t_{ij}) \in
  \mathfrak{T}_{U}^{\beta}(m)$, then $\lambda_{i} = t_{ii}^{2}$, $t_{ii} \geq 0$, $i = 1, 2,
  \ldots ,m$. See \citet[p. 931, first paragraph]{hl:01}, \citet[p. 390, lines -11 to
  -16]{hlz:05} and \citet[p.5, lines 1-6]{k:14}.
  \item if $\kappa = (p, \dots, p)$, then
    \begin{equation}\label{qk3}
        q_{\kappa}(\mathbf{A}) = |\mathbf{A}|^{p},
    \end{equation}
    in particular if $p=0$, then $q_{\kappa}(\mathbf{A}) = 1$.
  \item if $\tau = (t_{1}, t_{2}, \dots, t_{m})$, $t_{1}\geq t_{2}\geq \cdots \geq t_{m} \geq
  0$, then
    \begin{equation}\label{qk41}
        q_{\kappa+\tau}(\mathbf{A}) = q_{\kappa}(\mathbf{A})q_{\tau}(\mathbf{A}),
    \end{equation}
    in particular if $\tau = (p,p, \dots, p)$,  then
    \begin{equation}\label{qk42}
        q_{\kappa+\tau}(\mathbf{A}) \equiv q_{\kappa+p}(\mathbf{A}) = |\mathbf{A}|^{p} q_{\kappa}(\mathbf{A}).
    \end{equation}
    \item Finally, for $\mathbf{B} \in \mathfrak{T}_{U}^{\beta}(m)$  in such a manner that $\mathbf{C} =
    \mathbf{B}^{*}\mathbf{B} \in \mathfrak{S}_{m}^{\beta}$,
    \begin{equation}\label{qk5}
        q_{\kappa}(\mathbf{B}^{*}\mathbf{AB}) = q_{\kappa}(\mathbf{C})q_{\kappa}(\mathbf{A})
    \end{equation}
    and
    \begin{equation}\label{qk6}
        q_{\kappa}(\mathbf{B}^{*-1}\mathbf{A}\mathbf{B}^{-1}) = (q_{\kappa}(\mathbf{C}))^{-1}q_{\kappa}(\mathbf{A})
        = q_{-\kappa}(\mathbf{C})q_{\kappa}(\mathbf{A}),
    \end{equation}
see \citet[p. 776, eq. (2.1)]{hlz:08}.
\end{enumerate}
\begin{remark}
Let $\mathcal{P}(\mathfrak{S}_{m}^{\beta})$ denote the algebra of all polynomial functions on
$\mathfrak{S}_{m}^{\beta}$, and $\mathcal{P}_{k}(\mathfrak{S}_{m}^{\beta})$ the subspace of homogeneous
polynomials of degree $k$ and let $\mathcal{P}^{\kappa}(\mathfrak{S}_{m}^{\beta})$ be an irreducible
subspace of $\mathcal{P}(\mathfrak{S}_{m}^{\beta})$ such that
$$
  \mathcal{P}_{k}(\mathfrak{S}_{m}^{\beta}) = \sum_{\kappa}\bigoplus
  \mathcal{P}^{\kappa}(\mathfrak{S}_{m}^{\beta}).
$$
Note that $q_{\kappa}$ is a homogeneous polynomial of degree $k$, moreover $q_{\kappa} \in
\mathcal{P}^{\kappa}(\mathfrak{S}_{m}^{\beta})$, see \citet{gr:87}.
\end{remark}
In (\ref{gammagen1}), $[a]_{\kappa}^{\beta}$ denotes the generalised Pochhammer symbol of weight
$\kappa$, defined as
\begin{eqnarray*}
  [a]_{\kappa}^{\beta} &=& \prod_{i = 1}^{m}(a-(i-1)\beta/2)_{k_{i}}\\
    &=& \frac{\pi^{m(m-1)\beta/4} \displaystyle\prod_{i=1}^{m}
    \Gamma[a + k_{i} -(i-1)\beta/2]}{\Gamma_{m}^{\beta}[a]} \\
    &=& \frac{\Gamma_{m}^{\beta}[a,\kappa]}{\Gamma_{m}^{\beta}[a]},
\end{eqnarray*}
where $\re(a) > (m-1)\beta/2 - k_{m}$ and
$$
  (a)_{i} = a (a+1)\cdots(a+i-1),
$$
is the standard Pochhammer symbol.

An alternative definition of the generalised gamma function of weight $\kappa$ is proposed by
\citet{k:66}, which is defined as%
\begin{eqnarray}\label{int2}
  \Gamma_{m}^{\beta}[a,-\kappa] &=& \displaystyle\int_{\mathbf{A} \in \mathfrak{P}_{m}^{\beta}}
    \etr\{-\mathbf{A}\} |\mathbf{A}|^{a-(m-1)\beta/2 - 1} q_{\kappa}(\mathbf{A}^{-1})
    (d\mathbf{A}) \\
&=& \pi^{m(m-1)\beta/4}\displaystyle\prod_{i=1}^{m} \Gamma[a - k_{i}
    -(m-i)\beta/2] \nonumber\\ \label{gammagen2}
&=& \displaystyle\frac{(-1)^{k} \Gamma_{m}^{\beta}[a]}{[-a +(m-1)\beta/2
    +1]_{\kappa}^{\beta}} ,
\end{eqnarray}
where $\re(a) > (m-1)\beta/2 + k_{1}$.

Finally, the following Jacobians involving the $\beta$ parameter, reflects the generalised power of the
algebraic technique; the can be seen as extensions of the full derived and unconnected results in the
real, complex or quaternion cases, see \citet{fk:94} and \citet{dggj:11}. These results are the base for
several matrix and matric variate generalised analysis.

\begin{proposition}\label{lemlt}
Let $\mathbf{X}$ and $\mathbf{Y} \in \mathfrak{L}_{n,m}^{\beta}$  be matrices of functionally independent
variables, and let $\mathbf{Y} = \mathbf{AXB} + \mathbf{C}$, where $\mathbf{A} \in
\mathfrak{L}_{n,n}^{\beta}$, $\mathbf{B} \in \mathfrak{L}_{m,m}^{\beta}$ and $\mathbf{C} \in
\mathfrak{L}_{n,m}^{\beta}$ are constant matrices. Then
\begin{equation}\label{lt}
    (d\mathbf{Y}) = |\mathbf{A}^{*}\mathbf{A}|^{m\beta/2} |\mathbf{B}^{*}\mathbf{B}|^{
    mn\beta/2}(d\mathbf{X}).
\end{equation}
\end{proposition}

\begin{proposition}\label{lemhlt}
Let $\mathbf{X}$ and $\mathbf{Y} \in \mathfrak{S}_{m}^{\beta}$ be matrices of functionally independent
variables, and let $\mathbf{Y} = \mathbf{AXA^{*}} + \mathbf{C}$, where $\mathbf{A} \in
\mathfrak{L}_{m,m}^{\beta}$ and $\mathbf{C} \in \mathfrak{S}_{m}^{\beta}$ are constant matrices. Then
\begin{equation}\label{hlt}
    (d\mathbf{Y}) = |\mathbf{A}^{*}\mathbf{A}|^{(m-1)\beta/2+1} (d\mathbf{X}).
\end{equation}
\end{proposition}

\begin{proposition}[Singular Value Decomposition, $SVD$]\label{lemsvd}
Let $\mathbf{X} \in {\mathcal L}_{n,m}^{\beta}$  be matrix of functionally independent variables, such
that $\mathbf{X} = \mathbf{W}_{1}\mathbf{D}\mathbf{V}^{*}$ \ with \ $\mathbf{W}_{1} \in {\mathcal
V}_{m,n}^{\beta}$, $\mathbf{V} \in \mathfrak{U}^{\beta}(m)$ \ and \ $\mathbf{D} = \diag(d_{1},
\cdots,d_{m}) \in \mathfrak{D}_{m}^{1}$, $d_{1}> \cdots > d_{m} > 0$. Then
\begin{equation}\label{svd}
    (d\mathbf{X}) = 2^{-m}\pi^{\varrho} \prod_{i = 1}^{m} d_{i}^{\beta(n - m + 1) -1}
    \prod_{i < j}^{m}(d_{i}^{2} - d_{j}^{2})^{\beta} (d\mathbf{D}) (\mathbf{V}^{*}d\mathbf{V})
    (\mathbf{W}_{1}^{*}d\mathbf{W}_{1}),
\end{equation}
where
$$
  \varrho = \left\{
             \begin{array}{rl}
               0, & \beta = 1; \\
               -m, & \beta = 2; \\
               -2m, & \beta = 4; \\
               -4m, & \beta = 8.
             \end{array}
           \right.
$$
\end{proposition}

\begin{proposition}\label{lemW}
Let $\mathbf{X} \in \mathfrak{L}_{n,m}^{\beta}$  be matrix of functionally independent variables, and
write $\mathbf{X}=\mathbf{V}_{1}\mathbf{T}$, where $\mathbf{V}_{1} \in {\mathcal V}_{m,n}^{\beta}$ and
$\mathbf{T}\in \mathfrak{T}_{U}^{\beta}(m)$ with positive diagonal elements. Define $\mathbf{S} =
\mathbf{X}^{*}\mathbf{X} \in \mathfrak{P}_{m}^{\beta}.$ Then
\begin{equation}\label{w}
    (d\mathbf{X}) = 2^{-m} |\mathbf{S}|^{\beta(n - m + 1)/2 - 1}
    (d\mathbf{S})(\mathbf{V}_{1}^{*}d\mathbf{V}_{1}),
\end{equation}
\end{proposition}

\section{Matrix multivariate $T$-Riesz distribution}\label{sec3}

A detailed discussion of Riesz distribution may be found in \citet{hl:01} and \citet{dg:15a}. In addition
the Kotz-Riesz distribution is studied in detail in \citet{dg:15c}. For your convenience, we adhere to
standard notation stated in \citet{dg:15a, dg:15c}. Before, consider the following two definitions of
Kotz-Riesz and Riesz distributions.

From  \citet{dg:15c}.
\begin{definition}\label{defKR}
Let $\boldsymbol{\Sigma} \in \boldsymbol{\Phi}_{m}^{\beta}$, $\boldsymbol{\Theta} \in
\boldsymbol{\Phi}_{n}^{\beta}$, $\boldsymbol{\mu} \in \mathfrak{L}^{\beta}_{n,m}$ and  $\kappa = (k_{1},
k_{2}, \dots, k_{m}) \in \Re^{m}$. And let $\mathbf{Y} \in \mathfrak{L}^{\beta}_{n,m}$ and
$\u(\mathbf{B}) \in \mathfrak{T}_{U}^{\beta}(n)$, such that $\mathbf{B} =
\u(\mathbf{B})^{*}\u(\mathbf{B})$ is the Cholesky decomposition of $\mathbf{B} \in
\mathfrak{S}_{m}^{\beta}$.
\begin{enumerate}
  \item Then it is said that $\mathbf{Y}$ has a Kotz-Riesz distribution of type I and its density function is
  $$
    \frac{\beta^{mn\beta/2+\sum_{i = 1}^{m}k_{i}}\Gamma_{m}^{\beta}[n\beta/2]}{\pi^{mn\beta/2}
    \Gamma_{m}^{\beta}[n\beta/2,\kappa] |\boldsymbol{\Sigma}|^{n\beta/2}|\boldsymbol{\Theta}|^{m\beta/2}}
    \etr\left\{- \beta\tr \left [\boldsymbol{\Sigma}^{-1} (\mathbf{Y} - \boldsymbol{\mu})^{*}
    \boldsymbol{\Theta}^{-1}(\mathbf{Y} - \boldsymbol{\mu})\right ]\right\}
  $$
  \begin{equation}\label{dfEKR1}\hspace{3.1cm}
    \times q_{\kappa}\left [\u(\boldsymbol{\Sigma})^{*-1} (\mathbf{Y} - \boldsymbol{\mu})^{*}
    \boldsymbol{\Theta}^{-1}(\mathbf{Y} - \boldsymbol{\mu})\u(\boldsymbol{\Sigma})^{-1}\right ](d\mathbf{Y})
  \end{equation}
  with $\re(n\beta/2) > (m-1)\beta/2 - k_{m}$;  denoting this fact as
  $$
    \mathbf{Y} \sim \mathcal{KR}^{\beta, I}_{n \times m}
    (\kappa,\boldsymbol{\mu}, \boldsymbol{\Theta}, \boldsymbol{\Sigma}).
  $$
  \item Then it is said that $\mathbf{Y}$ has a Kotz-Riesz distribution of type II and its density function is
  $$
    \frac{\beta^{mn\beta/2-\sum_{i = 1}^{m}k_{i}}\Gamma_{m}^{\beta}[n\beta/2]}{\pi^{mn\beta/2}\Gamma_{m}^{\beta}[n\beta/2,-\kappa]
    |\boldsymbol{\Sigma}|^{n\beta/2}|\boldsymbol{\Theta}|^{m\beta/2}}
     \etr\left\{- \beta\tr \left [\boldsymbol{\Sigma}^{-1} (\mathbf{Y} - \boldsymbol{\mu})^{*}
    \boldsymbol{\Theta}^{-1}(\mathbf{Y} - \boldsymbol{\mu})\right ]\right\}
  $$
  \begin{equation}\label{dfEKR2}\hspace{2.5cm}
    \times q_{\kappa}\left [\left(\u(\boldsymbol{\Sigma})^{*-1} (\mathbf{Y} - \boldsymbol{\mu})^{*}
    \boldsymbol{\Theta}^{-1}(\mathbf{Y} - \boldsymbol{\mu})\u(\boldsymbol{\Sigma})^{-1/2}\right)^{-1}\right ](d\mathbf{Y})
  \end{equation}
  with $\re(n\beta/2) > (m-1)\beta/2 + k_{1}$;  denoting this fact as
  $$
    \mathbf{Y} \sim \mathcal{KR}^{\beta, II}_{n \times m}
    (\kappa,\boldsymbol{\mu}, \boldsymbol{\Theta}, \boldsymbol{\Sigma}).
  $$
\end{enumerate}
\end{definition}

From \citet{hl:01} and \citet{dg:15a}.
\begin{definition}\label{defR}
Let $\mathbf{\Xi} \in \mathbf{\Phi}_{m}^{\beta}$ and  $\kappa = (k_{1}, k_{2}, \dots, k_{m}) \in
\Re^{m}$.
\begin{enumerate}
  \item Then it is said that $\mathbf{V}$ has a Riesz distribution of type I if its density function is
  \begin{equation}\label{dfR1}
    \frac{\beta^{am+\sum_{i = 1}^{m}k_{i}}}{\Gamma_{m}^{\beta}[a,\kappa] |\mathbf{\Xi}|^{a}q_{\kappa}(\mathbf{\Xi})}
    \etr\{-\beta\mathbf{\Xi}^{-1}\mathbf{V}\}|\mathbf{V}|^{a-(m-1)\beta/2 - 1}
    q_{\kappa}(\mathbf{V})(d\mathbf{V})
  \end{equation}
  for $\mathbf{V} \in \mathfrak{P}_{m}^{\beta}$ and $\re(a) \geq (m-1)\beta/2 - k_{m}$;
  denoting this fact as $\mathbf{V} \sim \mathcal{R}^{\beta, I}_{m}(a,\kappa,
  \mathbf{\Xi})$.
  \item Then it is said that $\mathbf{V}$ has a Riesz distribution of type II if its density function is
  \begin{equation}\label{dfR2}
     \frac{\beta^{am-\sum_{i = 1}^{m}k_{i}}}{\Gamma_{m}^{\beta}[a,-\kappa]
   |\mathbf{\Xi}|^{a}q_{\kappa}(\mathbf{\Xi}^{-1})}\etr\{-\beta\mathbf{\Xi}^{-1}\mathbf{V}\}
  |\mathbf{V}|^{a-(m-1)\beta/2 - 1} q_{\kappa}(\mathbf{V}^{-1}) (d\mathbf{V})
  \end{equation}
  for $\mathbf{V} \in \mathfrak{P}_{m}^{\beta}$ and $\re(a) > (m-1)\beta/2 + k_{1}$;
  denoting this fact as $\mathbf{V} \sim \mathcal{R}^{\beta, II}_{m}(a,\kappa,
  \mathbf{\Xi})$.
\end{enumerate}
\end{definition}

This way, in this section, two versions of the matrix multivariate $T$-Riesz distribution and the
corresponding generalised matrix multivariate beta type II distributions are obtained.

\begin{theorem}\label{teo4}
Let $(S^{1/2})^{2} = S \sim \mathcal{R}_{1}^{\beta,I}(\nu\beta/2, k,\rho)$,  $\rho
> 0$, $k \in \Re$ and $\re(\nu\beta/2)> -k$; independent of $\mathbf{Y} \sim \mathcal{KR}_{n \times
m}^{\beta,I}(\tau,\mathbf{0}, \mathbf{\Theta}, \mathbf{\Sigma})$, $\mathbf{\Sigma} \in
\mathfrak{P}_{m}^{\beta}$, $\mathbf{\Theta} \in \mathfrak{P}_{n}^{\beta}$  and $\re([n\beta/2)>
(m-1)\beta/2-t_{m}$. In addition, define $\mathbf{T} = S^{-1/2}\mathbf{Y}+ \boldsymbol{\mu}$ with
$\boldsymbol{\mu} \in \mathcal{L}_{n,m}^{\beta}$ a constant matrix. Then the density of $\mathbf{T}$ is
$$
    \propto  \left[1+\rho\tr \mathbf{\Sigma}^{-1}(\mathbf{T}- \boldsymbol{\mu})^{*}
    \mathbf{\Theta}^{-1}(\mathbf{T}- \boldsymbol{\mu}) \right]^{-[(\nu +mn)\beta/2+k+\sum_{i=1}^{m}t_{i}]}
$$
\begin{equation}\label{mmtr1}
    \hspace{2cm} \times \ q_{\tau}\left(\mathbf{\u(\Sigma})^{*-1}(\mathbf{T}- \boldsymbol{\mu})^{*}
    \mathbf{\Theta}^{-1}(\mathbf{T}- \boldsymbol{\mu})\u(\mathbf{\Sigma})^{-1}\right)(d\mathbf{T})
\end{equation}
with constant of proportionality
$$
  \frac{\Gamma^{\beta}_{m}[n\beta/2]\Gamma^{\beta}_{1}\left[(\nu+mn)\beta/2+k+\sum_{i=1}^{m}t_{i}\right]
  \rho^{\beta mn/2 + \sum_{i=1}^{m}t_{i}}}{\pi^{\beta mn/2} \Gamma^{\beta}_{m}[n\beta/2,\tau]
  \Gamma^{\beta}_{1}[\nu\beta/2+k]|\mathbf{\Sigma}|^{\beta n/2}|\mathbf{\Theta}|^{\beta m/2}},
$$
which is termed the \emph{matrix multivariate $T$-Riesz type I distribution} and is denoted as
$\mathbf{T} \sim \mathcal{MTR}_{m \times n}^{\beta,I}(\nu,k,\tau,\rho,\boldsymbol{\mu}, \mathbf{\Sigma},
\mathbf{\Theta})$.
\end{theorem}
\begin{proof}
From definition \ref{defKR} and \ref{defR}, the joint density of $S$ and $\mathbf{Y}$ is
$$
  \propto s^{\beta \nu/2+k -1} \etr\left\{-\beta \left(s/\rho + \tr\mathbf{\Sigma}^{-1} \mathbf{Y}^{*}
  \mathbf{\Theta}^{-1}\mathbf{Y}\right)\right\}  q_{\tau}\left(\tr\mathbf{\Sigma}^{-1} \mathbf{Y}^{*}
  \mathbf{\Theta}^{-1}\mathbf{Y})(ds)(d\mathbf{Y}\right)
$$
where the constant of proportionality is
$$
  c = \frac{\beta^{\nu\beta/2+k}}{\Gamma_{1}^{\beta}[\nu\beta/2+k] \rho^{\nu\beta/2+k}}
      \ \cdot \ \frac{\beta^{mn\beta/2+\sum_{i=1}^{m}t_{i}} \ \Gamma_{m}^{\beta}[n\beta/2]}{\pi^{mn\beta/2}
      \Gamma_{m}^{\beta}[n\beta/2,\tau] |\mathbf{\Sigma}|^{n\beta/2}|\mathbf{\Theta}|^{m\beta/2}}.
$$
Taking into account that by (\ref{lt})
$$
  (ds)(d\mathbf{Y})= s^{\beta mn/2}(ds)(d\mathbf{T}),
$$
the desired result is obtained integrating with respect to $s$.
\end{proof}

Similarly is obtained:
\begin{theorem}\label{teo41}
Let $\mathbf{T} = S^{-1/2}\mathbf{Y}+ \boldsymbol{\mu} \in \mathcal{L}_{n,m}^{\beta}$ where
$(S^{1/2})^{2} = S \sim \mathcal{R}_{1}^{\beta,II}(\nu\beta/2, k,\rho)$,  $\rho
> 0$, $k \in \Re$ and $\re(\nu\beta/2)> k$; independent of $\mathbf{Y} \sim \mathcal{KR}_{n \times
m}^{\beta,II}(\tau,\mathbf{0}, \mathbf{\Sigma}, \mathbf{\Theta})$, $\mathbf{\Sigma} \in
\mathfrak{P}_{m}^{\beta}$, $\mathbf{\Theta} \in \mathfrak{P}_{n}^{\beta}$  and $\re([n\beta/2)>
(m-1)\beta/2-t_{1}$. Then the density of $\mathbf{T}$ is
$$
    \propto  \left[1+\rho\tr \mathbf{\Sigma}^{-1}(\mathbf{T}- \boldsymbol{\mu})^{*}
    \mathbf{\Theta}^{-1}(\mathbf{T}- \boldsymbol{\mu}) \right]^{-[(\nu +mn)\beta/2-k-\sum_{i=1}^{m}t_{i}]}
$$
\begin{equation}\label{mmtr2}
    \hspace{2cm} \times \ q_{\tau}\left[\left(\u(\mathbf{\Sigma})^{*-1}(\mathbf{T}- \boldsymbol{\mu})^{*}
    \mathbf{\Theta}^{-1}(\mathbf{T}- \boldsymbol{\mu})\u(\mathbf{\Sigma})^{-1}\right)^{-1}\right]
\end{equation}
with constant of proportionality
$$
  \frac{\Gamma^{\beta}_{m}[n\beta/2]\Gamma^{\beta}_{1}\left[(\nu+mn)\beta/2-k-\sum_{i=1}^{m}t_{i}\right]
  \rho^{\beta mn/2-\sum_{i=1}^{m}t_{i}}}{\pi^{\beta mn/2}\Gamma^{\beta}_{m}[n\beta/2,-\tau]
  \Gamma^{\beta}_{1}[\nu\beta/2-k]|\mathbf{\Sigma}|^{\beta n/2}|\mathbf{\Theta}|^{\beta m/2}},
$$
which is termed the \emph{matrix multivariate $T$-Riesz type II distribution} and is denoted as
$\mathbf{T} \sim \mathcal{MTR}_{m \times n}^{\beta,II}(\nu,k,\tau,\rho,\boldsymbol{\mu}, \mathbf{\Theta},
\mathbf{\Sigma})$.
\end{theorem}

Next we study the corresponding matrix multivariate beta type II distributions.

\begin{theorem}\label{teo7}
Define $\mathbf{F} = \mathbf{T}^{*}\mathbf{T} \in \mathfrak{P}_{m}^{\beta}$, with $n \geq m$ and observe
that
$$
   \mathbf{F} = S^{-1}\mathbf{Y}^{*}\mathbf{Y} = S^{-1}\mathbf{W}.
$$
\begin{enumerate}
  \item  If $\mathbf{T} \sim \mathcal{MT}_{n \times m}^{\beta,I}(\nu, k,\tau,\rho,\boldsymbol{0},\mathbf{I}_{n},
  \mathbf{\Sigma})$, then, under the conditions of Theorem \ref{teo4} we have that, $\mathbf{W}=\mathbf{Y}^{*}\mathbf{Y}
  \sim \mathcal{R}_{m}^{\beta,I}(n\beta/2,\tau,\mathbf{\Sigma})$, with $\re(n\beta/2) >(m-1)\beta/2-t_{m}$ and
  the density of $\mathbf{F}$ is,
  \begin{equation}\label{MMTR1}
    \propto |\mathbf{F}|^{(n-m+1)\beta/2-1} (1+\rho\tr\mathbf{\Sigma}^{-1}\mathbf{F})^{-\left[(mn+\nu)\beta/2+k+\sum_{i=1}^{m}t_{i}\right]}
    q_{\tau}(\mathbf{F})(d\mathbf{F}),
  \end{equation}
  with constant of proportionality
  $$
    \frac{\Gamma^{\beta}_{1}\left[(\nu+mn)\beta/2+k+\sum_{i=1}^{m}t_{i}\right] \rho^{\beta mn/2 + \sum_{i=1}^{m}t_{i}}}
    { \Gamma^{\beta}_{m}[n\beta/2,\tau] \Gamma^{\beta}_{1}[\nu\beta/2+k]|\mathbf{\Sigma}|^{n\beta/1}q_{\tau}(\mathbf{\Sigma})},
  $$
  where $\re([\nu\beta/2)> (m-1)\beta/2-k_{m}$ and $\re(m\beta/2)> (m-1)\beta/2-t_{m}$.
  $\mathbf{F}$ is said to have a \emph{matrix multivariate c-beta-Riesz type II distribution}.

  \item If $\mathbf{T} \sim \mathcal{MT}_{n \times m}^{\beta,II}(\nu, k,\tau,\rho,\boldsymbol{0},
  \mathbf{I}_{n},\mathbf{\Sigma})$, then, under the conditions of Theorem \ref{teo41} we obtain that,
  $\mathbf{W}=\mathbf{Y}^{*}\mathbf{Y} \sim \mathcal{R}_{m}^{\beta,II}(n\beta/2,\tau,\mathbf{\Sigma})$, with
  $\re(n\beta/2) >(m-1)\beta/2+t_{1}$ and the density of $\mathbf{F}$ is,
 \begin{equation}\label{MMTR2}
    \propto |\mathbf{F}|^{(n-m+1)\beta/2-1} (1+\rho\tr\mathbf{\Sigma}^{-1}\mathbf{F})^{-\left[(mn+\nu)\beta/2-k-
    \sum_{i=1}^{m}t_{i}\right]}  q_{\tau}(\mathbf{F}^{^{-1}})(d\mathbf{F}),
  \end{equation}
  with constant of proportionality
  $$
    \frac{\Gamma^{\beta}_{1}\left[(\nu+mn)\beta/2-k-\sum_{i=1}^{m}t_{i}\right] \rho^{\beta mn/2 -
    \sum_{i=1}^{m}t_{i}}}{ \Gamma^{\beta}_{m}[n\beta/2,-\tau] \Gamma^{\beta}_{1}[\nu\beta/2-k]
    |\mathbf{\Sigma}|^{n\beta/1}q_{\tau}(\mathbf{\Sigma}^{-1})},
  $$
  where $\re([\nu\beta/2)> (m-1)\beta/2+k_{1}$ and $\re(m\beta/2)> (m-1)\beta/2+t_{1}$.
  $\mathbf{F}$ is said to have a \emph{matrix multivariate k-beta-Riesz type II distribution}.
\end{enumerate}
\end{theorem}
\begin{proof}
The desired result follows from (\ref{mmtr1}) and (\ref{mmtr2}) respectively, by applying (\ref{w}) and
then (\ref{vol}).
\end{proof}

If in theorems in this section are defined $k =0$ and $\tau = (0, \dots,0)$, the results in
\citet{dggj:12} are obtained as particular cases. Also, in real case, when $k =0$ and $\tau = (0,
\dots,0)$ the results in Theorem \ref{teo7}.1 contain as particular case the results in \citet[Problem
3.18, p. 118]{m:82}.

\section{Singular value densities}\label{sec4}
In this section, the joint densities of the singular values of matrices $\mathbf{T}$ types I and II are
derived. In addition, and as a direct consequence, the joint densities of the eigenvalues of $\mathbf{F}$
types I and II are obtained for real normed division algebras.

\begin{theorem}
\begin{enumerate}
\item Let $\alpha_{1}, \dots, \alpha_{m}$, $\alpha_{1}> \cdots > \alpha_{m} > 0$, be the singular values of
the random matrix $\mathbf{T} \sim \mathcal{MTR}_{n \times m}^{\beta,I}(\nu, k, \tau, \rho, \mathbf{0},
\mathbf{I}_{n}, \mathbf{I}_{m})$. Then its joint density is
$$
  \propto \prod_{i=1}^{m} \left(\alpha_{i}^{2}\right)^{(n-m+1)\beta/2-1/2}
  \left(1+\rho\sum_{i=1}^{m}\alpha_{i}^{2}\right)^{-\left[(\nu +mn)\beta/2+k+\sum_{i=1}^{m}t_{i}\right]}
  \hspace{1cm}
$$
\begin{equation}\label{svMT1}
\hspace{4cm} \times \ \prod_{i<j}^{m}\left(\alpha_{i}^{2} - \alpha_{j}^{2}\right)^{\beta}
  \frac{C_{\tau}^{\beta}(\mathbf{D}^{2})}{C_{\tau}^{\beta}(\mathbf{I}_{m})} \left(\bigwedge_{i=1}^{m}d\alpha_{i}\right)
\end{equation}
where the constant of proportionality is
$$
  \frac{2^{m} \pi^{\beta m^{2}/2 + \varrho} \ \Gamma^{\beta}_{1}\left[(\nu+mn)\beta/2+k+\sum_{i=1}^{m}t_{i}\right]
  \rho^{\beta mn/2 + \sum_{i=1}^{m}t_{i}}}{\Gamma_{m}^{\beta}[\beta m/2]\Gamma^{\beta}_{m}[n\beta/2,\tau]
  \Gamma^{\beta}_{1}[\nu\beta/2+k]}.
$$
\item Let $\alpha_{1}, \dots, \alpha_{m}$, $\alpha_{1}> \cdots > \alpha_{m} > 0$, be the singular values
of the random matrix $\mathbf{T} \sim \mathcal{MTR}_{n \times m}^{\beta,II}(\nu,k,\tau,\rho,\mathbf{0},
\mathbf{I}_{n}, \mathbf{I}_{m})$. Then its joint density is
$$
  \propto \prod_{i=1}^{m} \left(\alpha_{i}^{2}\right)^{(n-m+1)\beta/2-1/2}
  \left(1+\rho\sum_{i=1}^{m}\alpha_{i}^{2}\right)^{-[(\nu +mn)\beta/2-k-\sum_{i=1}^{m}t_{i}]}
  \hspace{1cm}
$$
\begin{equation}\label{svMT2}
\hspace{4cm} \times \ \prod_{i<j}^{m}\left(\alpha_{i}^{2} - \alpha_{j}^{2}\right)^{\beta}
  \frac{C_{\tau}^{\beta}(\mathbf{D}^{-2})}{C_{\tau}^{\beta}(\mathbf{I}_{m})} \left(\bigwedge_{i=1}^{m}d\alpha_{i}\right)
\end{equation}
where the constant of proportionality is
$$
  \frac{2^{m} \pi^{\beta m^{2}/2 + \varrho} \ \Gamma^{\beta}_{1}\left[(\nu+mn)\beta/2-k-\sum_{i=1}^{m}t_{i}\right]
  \rho^{\beta mn/2 - \sum_{i=1}^{m}t_{i}}}{\Gamma_{m}^{\beta}[\beta m/2]\Gamma^{\beta}_{m}[n\beta/2,-\tau]
  \Gamma^{\beta}_{1}[\nu\beta/2-k]},
$$
\end{enumerate}
Where $\varrho$ is defined in Lemma \ref{lemsvd}, $\mathbf{D} = \diag(\alpha_{1},\dots, \alpha_{m})$, and
$C_{\kappa}^{\beta}(\cdot)$ denotes the zonal spherical functions or spherical polynomials, see
\citet{gr:87} and \citet[Chapter XI, Section 3]{fk:94}.
\end{theorem}
\begin{proof}
This follows immediately from (\ref{mmtr1}) and (\ref{mmtr2}) respectively, first using (\ref{svd}), then
applying (\ref{vol}) and observing that, from \citep[Equation 4.8(2) and Definition 5.3]{gr:87} and
\citet[Chapter XI, Section 3]{fk:94}, we have that for $\mathbf{L} \in \mathfrak{P}_{m}^{\beta}$,
$$
    C_{\tau}^{\beta}(\mathbf{L}) = C_{\tau}^{\beta}(\mathbf{I}_{m})\int_{\mathbf{H} \in \mathfrak{U}^{\beta}(m)}
     q_{\kappa}(\mathbf{HLH}^{*})(d\mathbf{H}),
$$
\end{proof}

Finally, observe that $\alpha_{i} = \sqrt{\eig_{i}(\mathbf{T}\mathbf{T}^{*})}$, where
$\eig_{i}(\mathbf{A})$, $i = 1, \dots, m$, denotes the $i$-th eigenvalue of $\mathbf{A}$. Let $\gamma_{i}
= \eig_{i}(\mathbf{T}\mathbf{T}^{*}) = \eig_{i}(\mathbf{F})$, observing that, for example, $\alpha_{i} =
\sqrt{\gamma_{i}}$. Then
$$
  \bigwedge_{i=1}^{m} d\alpha_{i} =  2^{-m} \prod_{i=1}^{m}
  \gamma_{i}^{-1/2} \bigwedge_{i=1}^{m} d\gamma_{i},
$$
the corresponding joint densities of $\gamma_{1}, \dots, \gamma_{m}$, $\gamma_{1} > \cdots > \gamma_{m}>
0$ types I and II, are obtained from (\ref{svMT1}) and (\ref{svMT1}) respectively as
\begin{enumerate}
\item
$$
  \propto \prod_{i=1}^{m} \gamma_{i}^{(n-m+1)\beta/2-1/2}
  \left(1+\rho\sum_{i=1}^{m}\gamma_{i}\right)^{-\left[(\nu +mn)\beta/2+k+\sum_{i=1}^{m}t_{i}\right]}
  \hspace{1cm}
$$
$$
\hspace{4cm} \times \ \prod_{i<j}^{m}\left(\gamma_{i} - \gamma_{j}\right)^{\beta}
  \frac{C_{\tau}^{\beta}(\mathbf{G})}{C_{\tau}^{\beta}(\mathbf{I}_{m})} \left(\bigwedge_{i=1}^{m}d\alpha_{i}\right)
$$
where the constant of proportionality is
$$
  \frac{\pi^{\beta m^{2}/2 + \varrho} \ \Gamma^{\beta}_{1}\left[(\nu+mn)\beta/2+k+\sum_{i=1}^{m}t_{i}\right]
  \rho^{\beta mn/2 + \sum_{i=1}^{m}t_{i}}}{\Gamma_{m}^{\beta}[\beta m/2]\Gamma^{\beta}_{m}[n\beta/2,\tau]
  \Gamma^{\beta}_{1}[\nu\beta/2+k]}.
$$
\item
$$
  \propto \prod_{i=1}^{m} \gamma_{i}^{(n-m+1)\beta/2-1/2}
  \left(1+\rho\sum_{i=1}^{m}\gamma_{i}\right)^{-[(\nu +mn)\beta/2-k-\sum_{i=1}^{m}t_{i}]}
  \hspace{1cm}
$$
$$
\hspace{4cm} \times \ \prod_{i<j}^{m}\left(\gamma_{i} - \gamma_{j}\right)^{\beta}
  \frac{C_{\tau}^{\beta}(\mathbf{G}^{-1})}{C_{\tau}^{\beta}(\mathbf{I}_{m})} \left(\bigwedge_{i=1}^{m}d\alpha_{i}\right)
$$
where the constant of proportionality is
$$
  \frac{2^{m} \pi^{\beta m^{2}/2 + \varrho} \ \Gamma^{\beta}_{1}\left[(\nu+mn)\beta/2-k-\sum_{i=1}^{m}t_{i}\right]
  \rho^{\beta mn/2 - \sum_{i=1}^{m}t_{i}}}{\Gamma_{m}^{\beta}[\beta m/2]\Gamma^{\beta}_{m}[n\beta/2,-\tau]
  \Gamma^{\beta}_{1}[\nu\beta/2-k]},
$$
\end{enumerate}
where $\mathbf{G} = \diag(\gamma_{1}, \dots,\gamma_{m})$.

\section{Conclusions}

Although during the 90's and 2000's were obtained important results in theory of random matrices
distributions, the  past 30 years have reached a substantial development. Essentially, these advances
have been archived through two approaches based on the \emph{theory of Jordan algebras} and the \emph{
theory of real normed division algebras}. A basic source of the mathematical tools of theory of random
matrices distributions under Jordan algebras can be found in \citet{fk:94}; and specifically, some works
in the context of theory of random matrices distributions based on Jordan algebras are provided in
\citet{m:94}, \citet{cl:96}, \citet{hl:01}, \citet{hlz:05}, \citet{hlz:08} and \citet{k:14} and the
references therein. Parallel results on theory of random matrices distributions based on real normed
division algebras have been also developed in random matrix theory and statistics, see \citet{gr:87},
\citet{f:05}, \citet{dggj:11}, \citet{dggj:13}, among others. In  addition, from mathematical point of
view, several basic properties of the matrix multivariate Riesz distribution under \emph{the structure
theory of normal $j$-algebras}  and under \emph{theory of Vinberg algebras} in place of Jordan algebras
have been studied, see \citet{i:00} and \citet{bh:09}, respectively.

The interest in these generalisations from a theoretical point of view becomes imminent, but from the
practical point of view, we most keep in mind the fact from \citet{b:02}, \emph{there is still no proof
that the octonions are useful for understanding the real world}. We can only hope that eventually this
question will be settled on one way or another. Also, for the sake of completeness, in the present
article the case of octonions is considered, but the veracity of the results obtained for this case can
only be conjectured; since there are still many problems under study in the context of the octonions.

For the sake of completeness, in the present article the case of octonions is considered, but the
veracity of the results obtained for this case can only be conjectured. Nonetheless, \citet[Section
1.4.5, pp. 22-24]{f:05} it is proved that the bi-dimensional density function of the eigenvalue, for a
Gaussian ensemble of a $2 \times 2$ octonionic matrix, is obtained from the general joint density
function of the eigenvalues for the Gaussian ensemble, assuming $m = 2$ and $\beta = 8$, see Section
\ref{sec2}. Moreover, as is established in \citet{fk:94} and \citet{S:97} the result obtained in this
article are valid for the \emph{algebra of Albert}, that is when hermitian matrices ($\mathbf{S}$) or
hermitian product of matrices ($\mathbf{X}^{*}\mathbf{X}$) are $3 \times 3$ octonionic matrices.

Finally, note that if in sections \ref{sec3} and \ref{sec4} is defined $\tau = (p,\dots,p)$ the
corresponding results for the matrix multivariate Kotz type distribution are obtained as particular case,
see \citet{fl:99}.

\section*{Acknowledgements}
This research work was partially supported by IDI-Spain, Grants No. MTM2011-28962. This paper was written
during J. A. D\'{\i}az-Garc\'{\i}a's stay as a visiting professor at the Department of Statistics and O.
R. of the University of Granada, Spain.

\bibliographystyle{plain}

\end{document}